\documentclass[smallextended,numbook,runningheads,final]{svjour3}
\smartqed
\usepackage{graphicx}
\usepackage{mathptmx}
\usepackage{amsmath}
\usepackage{amssymb}

\newcommand{\e}{\mathrm{e}}
\newcommand{\imm}{\mathrm{i}}
\newcommand{\diff}{\mathrm{d}}

\newcommand{\R}{\mathbb{R}}
\newcommand{\C}{\mathbb{C}}
\renewcommand{\v}[1]{\vec{#1}}
\newcommand{\diag}{\mathrm{diag}}
\renewcommand{\span}{\mathrm{span}}
\renewcommand{\S}{\mathcal{S}}
\newcommand{\Null}{\mathrm{Null}}
\newcommand{\mi}[1]{\text{\boldmath{$#1$}}}
\renewcommand{\ll}{\mbox{\mathversion{bold}$\ell$\mathversion{normal}}}
\newcommand{\sr}[1]{\textnormal{\footnotesize{#1}}}

\title{Fast transforms for high order boundary conditions\thanks{The
work was partially supported by MIUR, grant number and 2006017542}}
\author{Marco Donatelli}
\institute{M Donatelli \at
Dipartimento di Fisica e Matematica, Universit{\`a} dell'Insubria,
Via Valleggio 11, 22100 Como, Italy.\\
\email{marco.donatelli@uninsubria.it}
}

\begin{document}
\maketitle
\begin{abstract}
We study strategies for increasing the precision in the blurring
models by maintaining a complexity in the related numerical linear
algebra procedures (matrix-vector product, linear system solution,
computation of eigenvalues etc.) of the same order of the celebrated
Fast Fourier Transform. The key idea is the choice of a suitable
functional basis for representing signals and images.
Starting from an analysis of the spectral decomposition of blurring matrices
associated to the antireflective boundary conditions introduced in
[S. Serra Capizzano, SIAM J. Sci. Comput. 25-3 pp. 1307--1325],
we extend the model for preserving polynomials of higher degree
and fast computations also in the nonsymmetric case.

We apply the proposed model to Tikhonov regularization with smoothing norms and
the generalized cross validation for choosing the regularization parameter.
A selection of numerical experiments shows the effectiveness
of the proposed techniques.
\end{abstract}

\section{Introduction}\label{sect:intro}

We consider the de-convolution problem in the case of signals where
the convolution kernel is space invariant. In that case the observed
signal $g:{\cal I}\to\R$, ${\cal I}\subset\R$ is expressible as
\begin{equation}\label{eq:mod_cont}
    g(x) = \int_{\R} k(x-y)f(y)\,\diff y,
\end{equation}
where $f$ denotes the true signal. The approximation of the integral
operator via an elementary rectangle formula over an equispaced grid
with $n$ nodes leads to a linear system with $n$ equations.

When imposing proper boundary conditions, the related undetermined
linear system becomes square and invertible and fast filter
algorithms of Tikhonov type can be employed. When talking of fast
algorithms, given $n$ the size of the related matrices, we mean an
algorithm involving a constant number (independent of $n$) of fast
trigonometric transforms (Fourier, sine, cosine, Hartley
transforms) so that the overall cost is given by $O(n \log n)$
arithmetic operations.

For instance, when dealing with periodic boundary conditions, we
obtain circulant matrices which are diagonalizable by using the
celebrated fast Fourier transform (FFT). Unfortunately such boundary
conditions are not always satisfactory from the viewpoint of the
reconstruction quality. In fact, if the original signal is not
periodic the presence of ringing effects given by the periodic
boundary conditions spoils the precision of the reconstruction.

More accurate models are described by the reflective \cite{NCT} and
antireflective \cite{model-tau} boundary conditions, where the
continuity of the signal and of its derivative are imposed,
respectively. However the fast algorithms are applicable in this
context only when symmetric point spread functions (PSFs) are taken
into consideration.

The PSF represents the blur of a single pixel in the original
signal. Therefore, since it is reasonable to expect that the global
light intensity is preserved, the PSF is nothing else that a global
mask having nonnegative entries and total sum equal to $1$
(conservation law).  Often in several application such a PSF is
symmetric and consequently the symbol associated to its mask is an
even function.

Usually the antireflective boundary conditions lead to better
reconstructions since linear signals are reconstructed exactly, while
the periodic boundary conditions approximate badly a linear function
by a discontinuous one and the reflective ones by a piece-wise
linear function: in both the latter case Gibbs phenomena (called
ringing effects) are observed which are especially pronounced for
periodic boundary conditions.
The evidence of such fact is observed in several papers in the literature,
e.g. \cite{ADNS,ChHa08,DES,AR-reblur,Perr,model-tau}.

Such good behavior of the antireflective boundary conditions comes
directly from their definition \cite{model-tau}, since the
continuity of the first derivative of the signal was automatically
imposed. From an algebraic viewpoint, the latter property can be
derived from the spectral decomposition of the
coefficient matrix in the associated linear system. Indeed, when
considering a symmetric PSF and antireflective boundary conditions,
the linear system is represented by a matrix whose eigenvalues equal
to $1$ (the normalization condition of the PSF coming from the
conservation law) are associated to an eigenvector basis spanning
all linear functions sampled over a uniform grid with $n$ nodes, see
\cite{ADNS,ChHa08}. In \cite{ADNS}, such a remark has been the
starting point for defining and analyzing the antireflective
transform and for designing fast algorithms for the spectral
filtering of blurred and noisy signals. This algebraic
interpretation is useful because it can be used for proposing
generalizations that preserve the possibility of defining fast
algorithms, while increasing the expected reconstruction quality
especially when smooth or piece-wise smooth signals are considered.

In this paper, starting from the previous algebraic interpretation,
we define higher order boundary conditions. This can be obtained
by algebraically imposing that the spanning of quadratic or cubic
polynomials over a proper uniform gridding are eigenvectors
related to the normalized eigenvalue $1$.
Our proposal improves the antireflective model when the true signal
is regular enough close to the boundary.
Moreover, an important property of the proposed approach is that
it allows to define fast algorithms also in  the case
of nonsymmetric PSFs (such as the blurring caused by motion).
We note that reflective and antireflective boundary conditions can resort to
fast transforms only in the case of symmetric boundary conditions, while in the
case of nonsymmetric PSF we have fast transforms only for periodic
boundary conditions which usually provide poor restorations
for nonperiodic signals.

In general, if some information on the low frequencies of the signal
to be reconstructed are available, it is sufficient to impose such
sampled components as eigenvectors of the blurring operator related
to the eigenvalue $1$ (we recall that the global spectrum will have
$1$ as spectral radius). In such a way these component will be
maintained exactly by the filtering algorithms since they cut only
the spectral components related to small eigenvalues (somehow close
to zero) which are presumed to be essentially associated to the
noise. In reality, the noise by its random nature of its entries
will be decomposable essentially in high frequencies while the true
signal is supposed to be approximated in the complementary subspace
of low frequencies. Therefore, when applying filtering algorithms,
if the blurring operator has non-negligible eigenvalues associated
only to low frequencies (for instance low degree polynomials), then
the reconstruction of the signal will be reasonably good while the
noise will be efficiently reduced.

Given this general context, the present note is aimed to define
spectral decomposition of the blurring matrix such that the related
transform given by the eigenvectors is fast, the conditioning of the
transform is moderate (for such an issue in connection with the
antireflective transform see \cite{DH08}), and the low frequencies
are associated only to non-negligible eigenvalues.

The organization of the paper is as follows.  Section \ref{sec:bcs}
we introduce the deblurring problem investigating the spectral
decomposition of the coefficient matrix for the different kinds of
boundary conditions. In Section \ref{sec:hord} we define higher
order boundary conditions starting from the spectral decomposition
of the antireflective matrix. Such transforms are used in
Tikhonov-like procedures in Section \ref{sec:tik}.
Section \ref{sec:numexp} deals with a selection of
numerical tests on the de-convolution of blurred and noisy signals
and images. In Section \ref{sec:mD} the proposals are extended to a multi-dimensional
setting. Finally Section \ref{sec:concl} is devoted to concluding
remarks.

\section{Boundary conditions and associated coefficient matrices}\label{sec:bcs}
In this section we introduce the objects of our analysis and we
revisit the spectral decomposition of blurring matrices in the case
of periodic, reflective, and antireflective boundary conditions.

Let $\v{f} = (\dots, f_0, f_1, \dots, f_n, f_{n+1}, \dots)^T$ be the
true signal and $\{j\}_{j=1}^n$ the set of indexes in the field
of view. Given a PSF $\v{h} = (h_{-m}, \dots, h_0,
\dots, h_m)$, with $2m+1 \leq n$, we can associate to the PSF the
symbol
\begin{equation}\label{eq:symbol}
    z(t) = \sum_{j=-m}^mh_j\e^{\imm jt}, \qquad \imm=\sqrt{-1}.
\end{equation}

\subsection{Periodic and Reflective boundary conditions}
\paragraph{Periodic boundary conditions} are defined imposing
\[
f_{1-j}=f_{n+1-j} \qquad \mbox{and} \qquad f_{n+j}=f_j,
\]
for $j =1, \dots, n$. The blurring matrix associated to periodic boundary
conditions is diagonalized by the Fourier matrix
\[
     F_{ij}^{(n)} = \frac{1}{\sqrt{n}}
\exp\left({\frac{- \imm 2\pi(i-1)(j-1)}{n}}\right), \qquad
i,j=1,\dots,n.
\]
More precisely, the blurring matrix is
\begin{equation}\label{eq:permat}
        A_P = (F^{(n)})^H\diag(z(\v{x}))F^{(n)},
\end{equation}
where $x_i = 2(i-1)\pi/n$, for $i=1,\dots,n$. We note that the
eigenvalues $\lambda_i=z(x_i)$ can be easy computed by $\lambda_i =
[F^{(n)}(A_p\v{e}_1)]_i/[F^{(n)}\v{e}_1]_i$, where $\v{e}_1$ is the
first vector of the canonical base.

\paragraph{Reflective boundary conditions} are defined imposing
\[
f_{1-j}=f_j \qquad \mbox{and} \qquad f_{n+j}=f_{n+1-j},
\]
for $j = 1,\dots,n$. If the PSF is symmetric, i.e., $h_{-j}=h_j$, then
the blurring matrix associated to reflective boundary conditions is
diagonalized by the cosine transform (see \cite{NCT})
\[
    C_{ij}^{(n)} = \sqrt{\frac{2-\delta_{i,1}}{n}}
    \cos\left(\frac{(i-1)(2j-1)\pi}{2n}\right), \qquad
    i,j=1,\dots,n,
\]
where $\delta_{i,1}=1$ if $i=1$ and zero otherwise. More precisely,
the blurring matrix \nolinebreak is
\begin{equation}\label{eq:refm}
    A_R = (C^{(n)})^T\diag(z(\v{x}))C^{(n)},
\end{equation}
where $x_i = (i-1)\pi/n$, for $i=1,\dots,n$. Like for periodic
boundary conditions, the eigenvalues can be easy computed by
$\lambda_i = [C^{(n)}(A\v{e}_1)]_i/[C^{(n)}\v{e}_1]_i$.

\subsection{Antireflective boundary conditions}
\paragraph{ Antireflective boundary conditions} are defined imposing (see
\cite{model-tau})
\[
f_{1-j}=2f_1-f_{j+1} \qquad \mbox{and} \qquad f_{n+j}=2f_n-f_{n-j},
\]
for $j = 1,\dots,n$.

Let $Q$ be the sine transform matrix of order $n-2$ with entries
\[
  Q_{ij} = \sqrt{\frac{2}{n-1}}
            \sin\left(\frac{ij\pi}{n-1}\right), \qquad
    i,j=1,\dots,n-2.
\]
The antireflective transform of order $n$ can be defined by the
matrix (see \cite{ADNS})
\begin{equation}\label{eq:ART}
    T=\left[
          \begin{array}{c @{\quad}|@{\quad} c @{\quad}|@{\quad} c}
                   & \v{0} & \\
            \v{p} & Q & J\v{p} \\
             & \v{0} &  \\
          \end{array}
        \right],
\end{equation}
where
\[
    p_i= \sqrt{\frac{n(2n-1)}{6(n-1)}} \, \left(1- \frac{i-1}{n-1}\right),
\]
for $i=1,\dots,n$ and where the permutation matrix $J$ has
nontrivial entries $J_{i,n+1-i}=1$, $i=1,\dots,n$. We note that
$\|\v{p}\|_2=1$; moreover $J$ is often called flip matrix.

If the PSF is symmetric and $2m+1\leq n-2$, the spectral
decomposition of the coefficient matrix in the case of
antireflective boundary conditions is
\begin{equation}\label{eq:jcf}
A_A = T \,\diag(z(\v{y}))\, T^{-1},
\end{equation}
with $\v{y}$ defined as $y_i=(i-1)\pi/(n-1)$ for $i=1,\dots,n-1$ and
$y_n=0$. The eigenvalues of $A$ can be computed in $O(n\log n)$ real
operations resorting to the discrete sine transform (see
\cite{ADS08}).

Concerning the inverse antireflective transform $T^{-1}$, in
\cite{ADNS} we have given its expression and the resulting form is
analogous to that of the direct transform $T$. As a matter of fact,
given an algorithm for the direct transform, a procedure for
computing the inverse transform  needs only to have a fast way
for multiply $T^{-1}$ by a vector.

\begin{remark}\label{rem_Sc}
Observe that $\S_l=\span\{\v{p}, J\v{p}\}$ is the subspace spanned
by equi-spaced samplings of linear functions. In that case its
linear complement is given by
$\S_l^C=\span\{\sin(j\v{x})\}_{j=1}^{n-2}$, with $x_i =
(i-1)\pi/(n-1)$, $i=1,\dots,n$. Unfortunately such a linear
complement is not orthogonal and consequently the related transform
cannot be unitary, as long as we maintain such a trigonometric basis
useful for the fast computations. Up to standard normalization
factors this choice leads to the antireflective transform
\eqref{eq:ART}.
\end{remark}

\begin{remark}\label{rem_eig}
Implementing filtering methods, like Tikhonov, $\S_l$ is about fully
preserved since the associated eigenvalues are $z(0)=1$.
\end{remark}

\section{Higher order boundary conditions}\label{sec:hord}
Starting from Remarks \ref{rem_Sc} and \ref{rem_eig}, we define
higher order boundary conditions which represent the main contribution of
this work. The approach in Section \ref{sec:bcs} defines accurate
boundary conditions imposing a prescribed regularity to the true
signal $f$. The study of the spectral decomposition of the
associated coefficient matrices is a subsequent step for defining
fast and stable filtering methods. In this section, we define higher
order boundary conditions starting from the eigenspace, i.e., the
signal components, that we wish to preserve.

We start by imposing to preserve $\S_l$ and by suggesting other choices
for $\S_l^C$. By the way, the request of giving fast algorithms
suggests the use of a cosine or exponential basis in place
of that of sine functions, both for
the direct and inverse transforms. To preserve
polynomials of low degree and at the same time to resort to fast
trigonometric transforms, we need a transform with a structure
analogous to \eqref{eq:ART}. Therefore we need the cosine transform
and the Fourier matrix of order $n-2$. We define $F=F^{(n-2)}$ and
$C=C^{(n-2)}$, explicitly
\[
    C_{ij} = \sqrt{\frac{2-\delta_{i,1}}{n-2}}
    \cos\left(\frac{(i-1)(2j-1)\pi}{2n-4}\right)
\]
and
\[
F_{ij} = \frac{1}{\sqrt{n-2}} \exp\left({\frac{- \imm
2\pi(i-1)(j-1)}{n-2}}\right),
\]
for $i,j=1,\dots,n-2$.

We note that the first column of $C^T$ and of $F^H$ are a sampling of
the constant function. Hence the span of the columns of $C^T$ or
$F^H$ has a nontrivial intersection with $\mathcal{S}_l$. Accordingly,
we choose the two vectors for completing these trigonometric basis
as a uniform sampling of a quadratic function instead of a linear
function. More precisely, instead of $\mathcal{S}_l$ we consider
$\mathcal{S}_q=\span\{\v{q}, J\v{q}\}$, where $\v{q}$ is a uniform
sampling of a quadratic function in an interval that will be fixed
later.

The interval and the sampling grid for the basis functions of our
transform are fixed according to the following remark.

\begin{remark}
Up to normalization, the $j$th column of $Q$ is $\sin(j\v{x})$,
where $x_i = i\pi/(n-1)$, for $i=1,\dots,n-2$. Extending the
sampling grid such that the $j$th frequency is extended by
continuity, we add the grid points $x_0 = 0$ and $x_{n-1}=\pi$.
Since $\sin(jx_0)=\sin(jx_{n-1})=0$ for all $j$, we obtain exactly the
two zero vectors in the first and the last row of $T$, i.e., the
$(j+1)$th column of $T$ is the $j$th column of $Q$ extended in $x_0$ and
$x_n$, for $j=1,\dots,n-2$. The first and the last column of $T$ are the
sampling of linear functions at the same equispaced points $x_i \in [0, \, \pi]$,
$i=0,\dots,n-1$.
\end{remark}

\subsection{The case of symmetric PSF}
Firstly, we consider a symmetric PSF. In such case we can use the
cosine basis. Up to normalization, the $(j+1)$th column of $C^T$ is
$\cos(j\v{x})$, where $x_i=(2i-1)\pi/(2n-4)$, for $i=1,\dots,n-2$.
Extending the grid by continuity, we add $x_0=-\pi/(2n-4)$ and
$x_{n-1}=(2n-3)\pi/(2n-4)$. With this extended grid we can define
the basis functions as a $n$ points uniform sampling of the interval
\begin{equation}\label{eq:ints}
    [a, \; b] \, = \, \left[-\frac{\pi}{2n-4}, \;
\frac{(2n-3)\pi}{2n-4}\right],
\end{equation}
where the grid points are
\begin{equation}\label{eq:points}
x_i=(2i-1)\pi/(2n-4), \qquad i=0,\dots,n-1.
\end{equation}

We fix $\v{q} = \tilde{\v{q}}/\|\tilde{\v{q}}\|_2$, where
$[\tilde{\v{q}}]_{i+1} = (b-x_i)^2$, $i=0,\dots,n-1$. The fast
transform associated to $\mathcal{S}_q$ and $C^T$ can be defined as
follows:
\begin{equation}\label{eq:Tc}
    T_C=\left[
          \begin{array}{c @{\quad}|@{\quad} c @{\quad}|@{\quad} c}
                   & \v{c}_a^T & \\
            \v{q} & C^T & J\v{q} \\
             & \v{c}_b^T &  \\
          \end{array}
        \right],
\end{equation}
with $[\v{c}_a]_j=\sqrt{\frac{2-\delta_{j,1}}{n-2}}\cos((j-1)a)$ and
$[\v{c}_b]_j=\sqrt{\frac{2-\delta_{j,1}}{n-2}}\cos((j-1)b)=
(-1)^{j-1}[\v{c}_a]_j$ since $b = \pi-a$, for
$j=1,\dots,n-2$.

It remains to define the eigenvalues associated to $T_c$. Since we
want to preserve $\S_q$, similarly to what was done for $\v{p}$ in the
case of the antireflective boundary conditions, we associate to
$\v{q}$ and $J\v{q}$ the eigenvalue $z(0)=1$. Concerning the other
frequencies, since they are defined by the cosine transform, we
consider the eigenvalues of the reflective matrix in
\eqref{eq:refm}, but of order $n-2$.

In conclusion, for the case of a symmetric PSF, we define a new
blurring matrix using the following spectral decomposition
\begin{equation}\label{eq:Ac}
    A_C =T_C\diag(z(\v{x}))T_C^{-1},
\end{equation}
where $x_i = (i-2)\pi/(n-2)$, for $i=2,\dots,n-1$, and
$x_1=x_n=0$.

We note that $z(x_1)=z(x_n)=1$, while the eigenvalues $z(x_i)$, for
$i=2,\dots,n-1$, are the same of $A_R$ of order $n-2$ and hence they
can be computed in $O(n\log n)$ by a discrete cosine transform. The
product of $T_C$ by a vector can be computed mainly resorting to the
inverse discrete cosine transform. The inverse of $T_C$ will be
studied in Subsection \ref{sec:sherm}, where we will show that the
product of $T_C^{-1}$ by a vector can be computed mainly resorting
to a discrete cosine transform. Therefore the spectral decomposition
\eqref{eq:Ac} can be used to define fast filtering methods in the
case of symmetric PSFs. Moreover, we expect an improved restoration
with respect to the antireflective model since $A_C$ preserves
uniform samplings of quadratic functions while $A_R$ preserves
only uniform samplings of linear functions.

\subsection{The case of nonsymmetric PSF}
In the case of nonsymmetric PSF we can use the exponential basis.
Up to normalization, the $(j+1)$th column of $F^H$ is
$\exp(\imm j\v{x})$, where $\imm=\sqrt{-1}$ and
$x_i=(i-1)2\pi/(n-2)$, for $i=1,\dots,n-2$.
Extending the grid by continuity, we add $x_0=-2\pi/(n-2)$ and
$x_{n-1}=2\pi$. With this extended grid, we can define
the basis functions as a $n$ points uniform sampling of the interval
\begin{equation}\label{eq:intn}
    [a, \; b] \, = \, \left[-2\pi/(n-2), \; 2\pi\right],
\end{equation}
where the grid points are
\begin{equation}\label{eq:pointn}
x_i=(i-1)2\pi/(n-2), \qquad i=0,\dots,n-1.
\end{equation}

We note that the interval and the grid points in the nonsymmetric case are
different with respect to the symmetric case (compare \eqref{eq:intn} with
\eqref{eq:ints} and \eqref{eq:pointn} with \eqref{eq:points}). Therefore, defining
$\v{q} = \tilde{\v{q}}/\|\tilde{\v{q}}\|_2$, where
$[\tilde{\v{q}}]_{i+1} = (b-x_i)^2$, $i=0,\dots,n-1$, it is different
from which obtained in the symmetric case in the previous subsection.
The fast transform associated to $\mathcal{S}_q$ and $F^H$ can be defined as
follows
\begin{equation}\label{eq:Tf}
    T_F=\left[
          \begin{array}{c @{\quad}|@{\quad} c @{\quad}|@{\quad} c}
                   & \v{c}_a^T & \\
            \v{q} & F^H & J\v{q} \\
             & \v{c}_b^T &  \\
          \end{array}
        \right],
\end{equation}
where $[\v{c}_a]_j=\exp(\imm(j-1)a)/\sqrt{n-2}$ and
$[\v{c}_b]_j=\exp(\imm(j-1)b)/\sqrt{n-2}=1/\sqrt{n-2}$, for
$j=1,\dots,n-2$.

It remains to define the eigenvalues associated to $T_F$.
Similarly to what done for $T_C$, we associate to
$\v{q}$ and $J\v{q}$ the eigenvalue $z(0)=1$, while for the other
frequencies we consider the eigenvalues of the circulant matrix in
\eqref{eq:permat}, but of order $n-2$.

Consequently, in the case of a generic PSF, we define a new
blurring matrix using the following spectral decomposition
\begin{equation}\label{eq:Af}
    A_F =T_F\diag(z(\v{x}))T_F^{-1},
\end{equation}
where $x_i = (i-2)2\pi/n$, for $i=2,\dots,n-1$, and $x_1=x_n=0$.

We note that $z(x_1)=z(x_n)=1$, while the eigenvalues $z(x_i)$, for
$i=2,\dots,n-1$, are the same of $A_P$ of order $n-2$ and hence they
can be computed in $O(n\log n)$ by a fast Fourier transform. The
product of $T_F$ by a vector can be computed essentially resorting to the
inverse fast Fourier transform. The inverse of $T_F$ will be
studied together with the inverse of $T_C$ in the next Subsection, where we will show that the
product of $T_F^{-1}$ by a vector can be computed by using
the fast Fourier transform. Therefore the spectral decomposition
\eqref{eq:Af} can be used to define fast filtering methods also in the
case of nonsymmetric PSFs.

\subsection{The inverse transform}\label{sec:sherm}
In this subsection we show that the inverse of $T_C$ and the inverse of $T_F$ are fast transforms.
This means that the associated matrix vector product can be performed mainly via a suitable
trigonometric transform.

\begin{theorem}\label{th:inv}
Let
\begin{equation}\label{eq:Tx}
    T_X=\left[
          \begin{array}{c @{\quad}|@{\quad} c @{\quad}|@{\quad} c}
                   & \v{c}_a^T & \\
            \v{q} & X^{-1} & J\v{q} \\
             & \v{c}_b^T &  \\
          \end{array}
        \right]
\end{equation}
be a given $n \times n$ matrix,
where $J$ is the flip matrix and $X$ is a discrete trigonometric transform such that $JXJ=X$.
Then $T_X^{-1}\v{y}$ can be computed in $O(n\log(n))$ for all $\v{y}\in\C^n$.
\end{theorem}
\begin{proof}
We note that
\begin{equation}\label{eq:split}
    T_X = \widetilde{T}_X + [\v{e}_1 \, | \, \v{e}_n ]
        \left[
        \begin{array}{ccc}
          0 &\v{c}_a^T & 0\\
          0 &\v{c}_b^T & 0
        \end{array}
        \right],
\end{equation}
where
\[ \widetilde{T}_X=\left[
          \begin{array}{c @{\quad}|@{\quad} c @{\quad}|@{\quad} c}
                   & \v{0}^T & \\
            \v{q} & X^{-1} & J\v{q} \\
             & \v{0}^T &  \\
          \end{array}
        \right]
\]
is easy to invert.
Hence $T_X^{-1}$ can be computed by the Sherman-Morrison-Woodbury formula.

We compute $\widetilde{T}_X^{-1}$. Since $\v{q}_n=0$ the first and the last row
can be decoupled and we look for $\widetilde{T}_X^{-1}$ of the form
\[ \widetilde{T}_X^{-1}=\left[
          \begin{array}{c @{\quad}|@{\quad} c @{\quad}|@{\quad} c}
            \alpha  & \v{0}^T & 0\\
            \v{v} & X & J\v{v} \\
            0  & \v{0}^T & \alpha \\
          \end{array}
        \right],
\]
Fixing $\v{q}=[q_1, \, \hat{\v{q}}^T, \, 0]^T$,
by direct computation $\alpha=1/q_1$ and $\v{v}=-X\hat{\v{q}}/q_1$.
Therefore, $\v{v}$ can be computed in $O(n\log n)$ by a trigonometric transform.
For the implementation it can be explicitly computed and inserted into the code.

Given $A\in \C^{n\times n}$ and $U, V \in \C^{n\times k}$, the Sherman-Morrison-Woodbury
formula is \cite{Gene}:
\begin{equation}\label{eq:smw}
    (A+UV^H)^{-1}=A^{-1}-A^{-1}U(I+V^HA^{-1}U)^{-1}V^HA^{-1}.
\end{equation}
It can be very useful for computing the inverse of $A+UV^H$ when $k \ll n$,
taking into account the possible instability.
Applying the formula \eqref{eq:smw} to \eqref{eq:split} we obtain
\begin{eqnarray*}
T_X^{-1} & = & \widetilde{T}_X^{-1}-\widetilde{T}_X^{-1}[\v{e}_1 \, | \, \v{e}_n ]
    \left(I+\left[
        \begin{array}{ccc}
          0 &\v{c}_a^T & 0\\
          0 &\v{c}_b^T & 0
        \end{array}
        \right]\widetilde{T}_X^{-1}[\v{e}_1 \, | \, \v{e}_n ]\right)^{-1}
        \left[
        \begin{array}{ccc}
          0 &\v{c}_a^T & 0\\
          0 &\v{c}_b^T & 0
        \end{array}
        \right]\widetilde{T}_X^{-1}\\
  & = & \widetilde{T}_X^{-1}-\left[
        \begin{array}{cc}
          \alpha & 0\\
          \v{v} & J\v{v} \\
          0 & \alpha
        \end{array}
        \right]
    \left[
        \begin{array}{cc}
          1+\v{c}_a^T\v{v} & \v{c}_a^TJ\v{v}\\
          \v{c}_b^T\v{v} & 1+\v{c}_b^TJ\v{v}
        \end{array}
    \right]^{-1}
    \left[
        \begin{array}{ccc}
          \v{c}_a^T\v{v} &\v{c}_a^TX & \v{c}_a^TJ\v{v}\\
          \v{c}_b^T\v{v} &\v{c}_b^TX & \v{c}_b^TJ\v{v}
        \end{array}
    \right].
\end{eqnarray*}
We note that $\v{c}_a^TX$ and $\v{c}_b^TX$ can be computed in $O(n\log(n))$
and moreover they can be explicitly computed and
inserted into the implementation like done for the vector $\v{v}$.
In this way the matrix vector product for $T_X^{-1}$ requires a fast discrete
trigonometric transform of $O(n\log(n))$ plus few lower order operations between vectors.
\qed
\end{proof}

From Theorem \ref{th:inv}, it follows that the product of $T_C^{-1}$ and $T_F^{-1}$,
by a vector can be computed in $O(n \log(n))$ and hence they are fast transforms.

\section{Tikhonov regularization with fast transforms}\label{sec:tik}
We consider the Tikhonov regularization, where the regularized
solution is computed as the solution of the following minimization
problem
\begin{equation}\label{eq:mintik}
    \min_{\v{f}\in\R^n}\left\{\|\v{g}-A\v{f}\|_2^2+\mu\,\|L\v{f}\|_2^2\right\}, \qquad \mu>0,
\end{equation}
where, $\mu$ is the properly chosen regularization parameter,
$\v{g}$ is the observed signal, $A$ is the coefficient matrix
and $L$ is a matrix such that ${\rm Null}(A) \bigcap {\rm
Null}(L) = 0$ (see \cite{EHN96}). The matrix $L$ is usually the
identity matrix or an approximation of partial derivatives.

It is convenient to define $L$ using the same boundary
conditions of $A$ in order to obtain fast algorithms.
For instance, $L$ equal to the Laplacian with antireflective boundary conditions is
\begin{equation}\label{eq:L}
    L_A = \left[
      \begin{array}{ccccc}
        0 &  & \dots & &0 \\
        -1 & 2 & -1 & & \\
        & \ddots & \ddots & \ddots &\\
        &  & -1 & 2 & -1\\
        0 &  & \dots & &0 \\
      \end{array}
    \right].
\end{equation}
We note that $\dim(\Null(L))=2$. However, ${\Null}(A) \bigcap {\Null}(L) = 0$
because $\Null(L)=\S_l$. We have
$L_A= T_X \,\diag(s(\v{y}))\, T_X^{-1}$, for $s(x) = (2-2\cos(x))$
and $\v{y}$ defined according to \eqref{eq:jcf}
(note that $s(y_1)=s(y_n)=s(0)=0$).

Using the approach in Section \ref{sec:hord}, for high order boundary conditions
the Laplacian matrix can be defined similarly by
\begin{eqnarray*}
  L_C &=& T_C\diag(s(\v{y}))T_C^{-1}, \\
  L_F &=& T_F\diag(s(\v{x}))T_F^{-1},
\end{eqnarray*}
where the grid points $\v{y}$ and $\v{x}$ are defined according to \eqref{eq:Ac}
and \eqref{eq:Af} respectively.

\subsection{Tikhonov regularization and reblurring}
The minimization problem \eqref{eq:mintik} is equivalent to the
normal equations
\begin{equation}\label{eq:netik}
    (A^TA+\mu L^TL)\v{f}=A^T\v{g}.
\end{equation}
Regarding the antireflective algebra, in \cite{AR-reblur} it was
observed that the transposition operation destroys the algebra
structure and leads to worse restorations with respect to reflective
boundary conditions. To overcome this problem, in \cite{DES} the authors
proposed the reblurring which replaces the transposition
with the correlation operation. Moreover, it was shown that
the latter is equivalent to compute the solution of a discrete problem
obtained by a proper discretization of a continuous regularized
problem. Practically, we replace $A^T$ with $A'$ obtained imposing
the same boundary conditions to the coefficient matrix arising from
the PSF rotated by 180 degrees. The matrices defined in \eqref{eq:ART},
\eqref{eq:Ac}, and \eqref{eq:Af} can be denoted by $A_X(z)$,
$X\in\{A,C,F\}$ since they are univocally defined by the function $z$
when the transform $T_X$ is fixed. With this notation $A'_X(z)=A_X(\bar{z})$
since the rotation of the PSF exchange $h_j$ with $h_{-j}$ in \eqref{eq:symbol},
which corresponds to take $\bar{z}$.
Therefore, in the case of periodic boundary conditions $A'=A^T$,
but this is not true in general for the other boundary conditions.
If the PSF is symmetric then $A'=A$ for every boundary conditions.

In the following we use the reblurring approach and hence
we replace \eqref{eq:netik} with
\begin{equation}\label{eq:nerebl}
    (A'A+\mu L'L)\v{f}_{\rm reg}=A'\v{g}.
\end{equation}
If we use the same boundary conditions for $A$ and $L$, the \eqref{eq:nerebl}
can be written as $A_X(|z|^2+\mu |s|^2)\v{f}_{\rm reg}=A_X(\bar{z})\v{g}$.
In \cite{DH08} it is proved that for antireflective
boundary conditions \eqref{eq:nerebl} defines a regularization
method when $L=I$ and the PSF is symmetric.

If the spectral decomposition of $A$ is
$A = T_XDT_X^{-1}$, where $D=\diag(\v{d})$, and $L=T_X\diag(\v{s})T_X^{-1}$,
then the spectral filter solution in \eqref{eq:nerebl} is given by
\begin{equation}\label{eq:freg}
\v{f}_\sr{reg} =
  T_X \Phi D^{-1} T_X^{-1}\,\v{g},\qquad
  \Phi=\mbox{diag}_{i=1,\dots,n}\left(\frac{|d_i|^2}{|d_i|^2 + \mu |s_i|^2}\right)
\end{equation}

\subsection{GCV for the estimation of the regularization parameter}
A largely used method for estimating the regularization parameter $\mu$ is the
generalized cross validation (GCV) \cite{GHW79}.
For the method in \eqref{eq:freg}, GCV determines the regularizing parameter $\mu$
that minimizes the GCV functional
\begin{equation}\label{eq:gcv}
G(\mu) = \frac{\|\v{g}-A\v{f}_{\sr{reg}}\|_2^2}{{\rm trace}(I - A T_X \Phi D^{-1} T_X^{-1})^2},
\end{equation}
where $\v{f}_{\sr{reg}}$ is defined in \eqref{eq:freg}.
For $A_P$ (periodic boundary conditions) and $A_R$ (reflective boundary conditions with a symmetric PSF),
the equation \eqref{eq:nerebl} is exactly equation \eqref{eq:netik}. In this case,
in \cite{HaNaOL06} it is proven that
\begin{equation}\label{eq:gcv2}
    G(\mu) = \frac{\sum_{i=1}^n(\sigma_i\hat{g}_i)^2}{\left(\sum_{i=1}^n{\sigma_i}\right)^2},
\end{equation}
where $\sigma_i = |s_i|^2 / (|d_i|^2+\mu |s_i|^2)$ and $\hat{\v{g}}=T_X^{-1}\v{g}$,
for $T_X^{-1}$ equal to $F^{(n)}$ and $C^{(n)}$, respectively.

Here we have
\begin{eqnarray}\label{eq:gcnum}
  \|\v{g}-A\v{f}_{\textnormal{\footnotesize{reg}}}\|_2 &=& \|T_X(I-\Phi)T_X^{-1}\v{g}\|_2 \nonumber\\
   & \approx & \|(I-\Phi)T_X^{-1}\v{g}\|_2,
\end{eqnarray}
because $T_X$ is not unitary but it is ``close'' to a unitary matrix since
it is a rank four correction of a unitary matrix. For the estimation
of the SVD of $T$ (antireflective boundary conditions case) see \cite{DH08}.
Therefore, we compute the regularization parameter $\mu$ by minimizing
the same functional as in \eqref{eq:gcv2}.
More precisely
\begin{equation}\label{eq:mugcv}
    \mu_{\rm GCV} = {\rm arg}\min_{\mu>0} \frac{\sum_{i=1}^n(\sigma_i\hat{g}_i)^2}{\left(\sum_{i=1}^n{\sigma_i}\right)^2}.
\end{equation}

\section{Numerical experiments}\label{sec:numexp}
We present some signal deblurring problems.
The restorations are obtained by employing Tikhonov regularization using \eqref{eq:freg}
with smoothing operator $L$ equal to the Laplacian.
The code is implemented in Matalab 7.0.

In the first example the observed signal is affected by a Gaussian blur and $0.1 \%$ of
Gaussian noise. True and observed signals are shown in Figure \ref{fig:signal-1Dsimm}.
\begin{figure}
\centering
\includegraphics[width = 5.8cm]{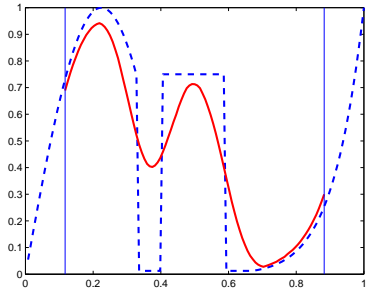}
\caption{- - - true signal, --- observed signal with
Gaussian blur and 0.1\% of noise. The vertical lines denote the field of view.}
\label{fig:signal-1Dsimm}
\end{figure}
We consider a low level of noise because in such case the restoration error is mainly
due to the error of the boundary conditions model.
Since the PSF is symmetric, we compare our blurring matrix $A_C$ with
reflective and antireflective boundary conditions.

Let $\v{f}$ be the true signal, the relative restoration errors (RRE)
$\|\v{f}-\v{f}_{\rm reg}\|_2/\|\v{f}\|_2$ is plotted in Figure \ref{fig:rre1Dsimm}.
In such figure it is evident that $A_C$ provides restorations with a lower RRE
with respect to antireflective boundary conditions, which are already known to be
more precise than reflective boundary conditions.
Moreover, the RRE curve varying the regularization parameter $\mu$ is flatter with
respect to the other boundary conditions. This allows a better estimation of the
regularization parameter using the GCV. The value $\mu_{\rm GCV}$ that
gives the minimum of the GCV functional in \eqref{eq:gcv2} is reported in Figure
\ref{fig:rre1Dsimm} with a `*'. It is evident that in the case of $A_C$ the RRE
obtained with $\mu_{\rm GCV}$ is closer to the minimum with respect to antireflective
boundary conditions.
The minimum RRE is $0.135$ for $A_C$ while it is $0.177$ for the antireflective boundary conditions.
Moreover, for  $A_C$ we obtain $\mu_{\rm GCV}=5.6\times 10^{-5}$ which gives a RRE equal to $0.135$,
while for antireflective boundary conditions $\mu_{\rm GCV}=1.6\times 10^{-4}$ gives a RRE equal to $0.502$.
\begin{figure}
\centering
\includegraphics[width = 6cm]{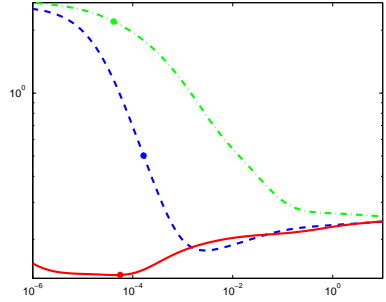}
\caption{RRE: --- $A_C$, - - - antireflective, - $\cdot$ - reflective
    ($*$ denotes values corresponding to $\mu_{\rm GCV}$).}
\label{fig:rre1Dsimm}
\end{figure}

The quality of the restoration is validated also from the
visual evidence of the restored signals. In Figure \ref{fig:rest-1Dsimm} we
show the restored signal corresponding to $\mu_{\rm opt}$, which is the value
of the regularization parameter $\mu$ corresponding to the minimum RRE, and to $\mu_{\rm GCV}$.
We note that  $A_C$ gives a better restoration
especially for preserving jumps in the signal. On the other hand this implies a slightly
lose of the smoothness of the restored signal. Eventually,
using $\mu_{\rm GCV}$ our proposal with $A_C$ gives a good enough restoration while this is not true
for the antireflective boundary conditions.
\begin{figure}
\centering
\begin{tabular}{c@{\!\!\!}c}
\includegraphics[width = 5.8cm]{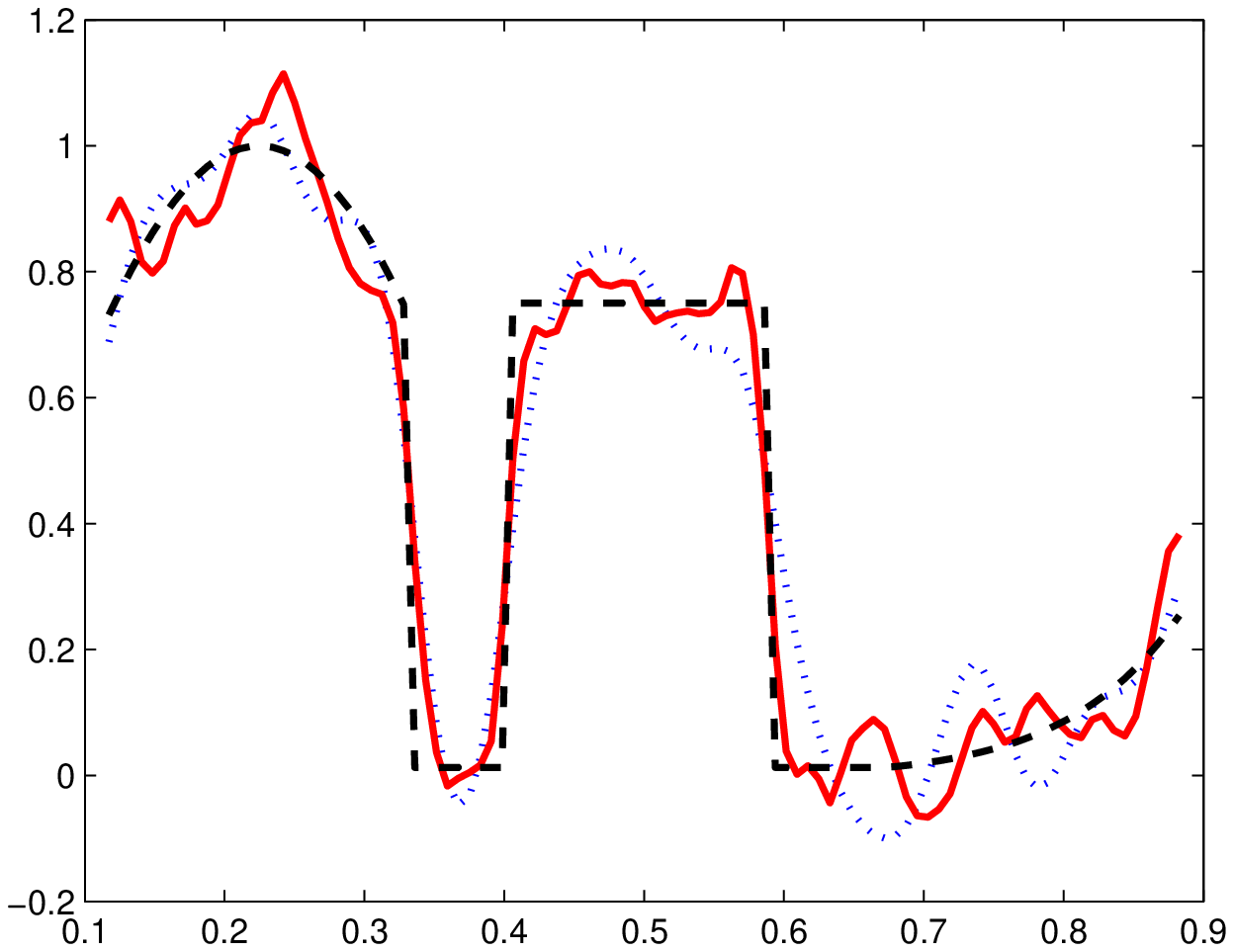}
& \includegraphics[width = 5.8cm]{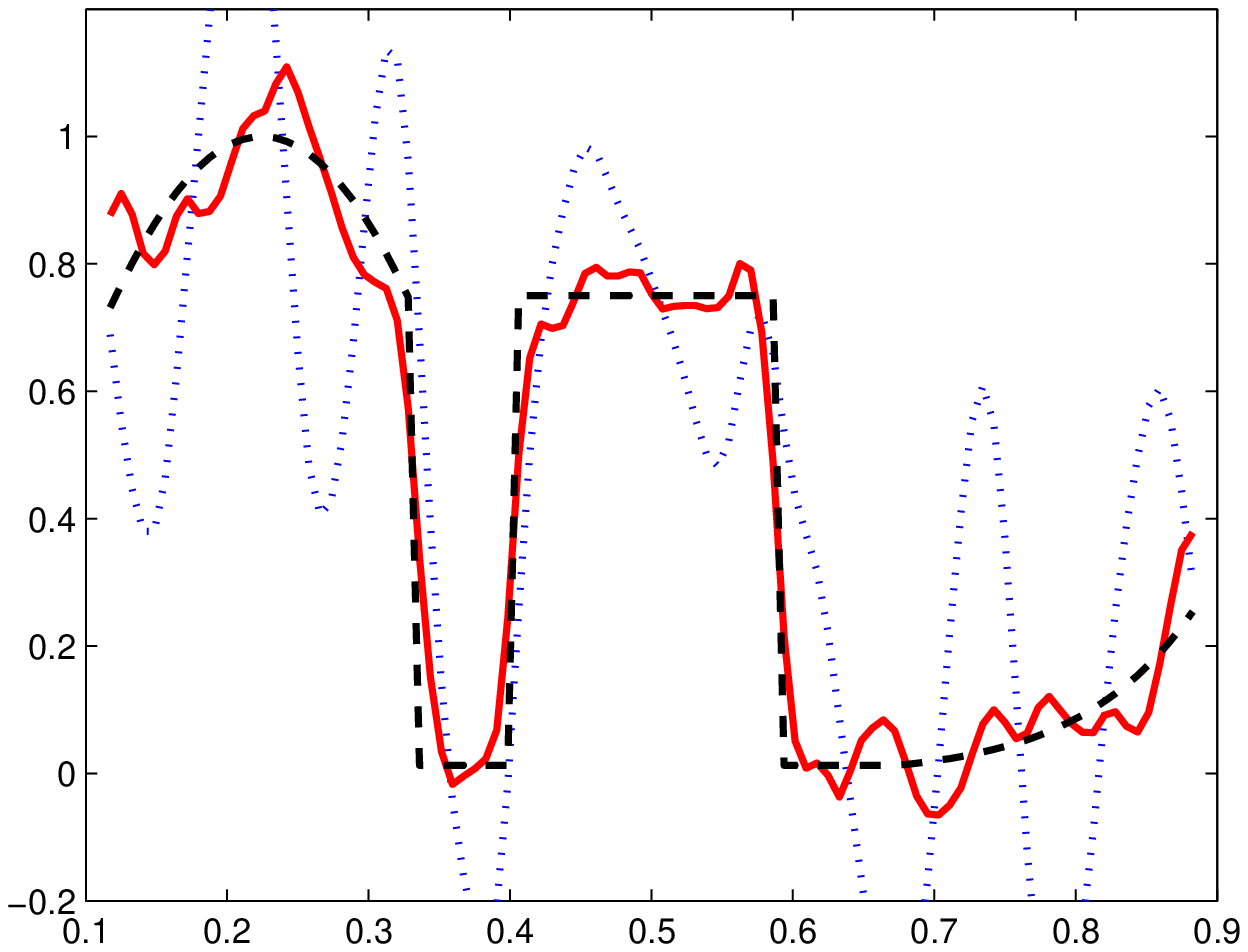}\\
$\mu_{\rm opt}$ & $\mu_{\rm GCV}$
\end{tabular}
\caption{Restored signals:
        --- $A_C$, $\cdot \cdot \cdot$ antireflective, - - - true signal.}
\label{fig:rest-1Dsimm}
\end{figure}

The second example is a moving PSF with a $1\%$ of Gaussian noise.
True and observed signals are shown in Figure \ref{fig:signal-1Dnons}.
\begin{figure}
\centering
\includegraphics[width = 5cm]{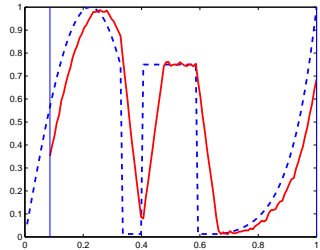}
\caption{--- true signal, - - - observed signal with
moving blur and 1\% of noise. The vertical lines denote the field of view. }
\label{fig:signal-1Dnons}
\end{figure}
Since the PSF is nonsymmetric we consider $A_F$ instead of $A_C$. Moreover, since
antireflective and reflective boundary conditions lead to matrices that can not be diagonalized
by fast transforms, we can compare $A_F$ only with periodic boundary conditions.
From Figure \ref{fig:rre1Dnons} and \ref{fig:rest-1Dnons} we note that the same considerations
done in the previous example hold unchanged.
Indeed the minimum RRE is $0.091$ for $A_F$ while it equals $0.198$ for the periodic boundary conditions.
Moreover, for  $A_F$ we obtain $\mu_{\rm GCV}=1.7\times 10^{-3}$ which gives a RRE equal to $0.096$,
while for periodic boundary conditions $\mu_{\rm GCV}=2.7\times 10^{-3}$ gives a RRE equal to $0.704$.

\begin{figure}
\centering
\includegraphics[width = 6cm]{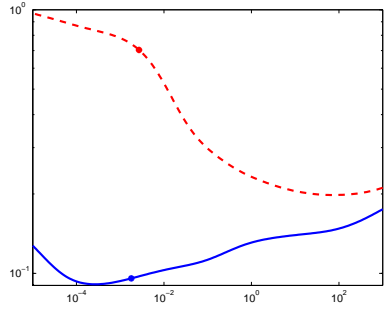}
\caption{RRE: --- $A_F$, - - - periodic
    ($*$ denotes values corresponding to $\mu_{\rm GCV}$).}
\label{fig:rre1Dnons}
\end{figure}
\begin{figure}
\centering
\begin{tabular}{c@{\!\!\!}c}
\includegraphics[width = 5cm]{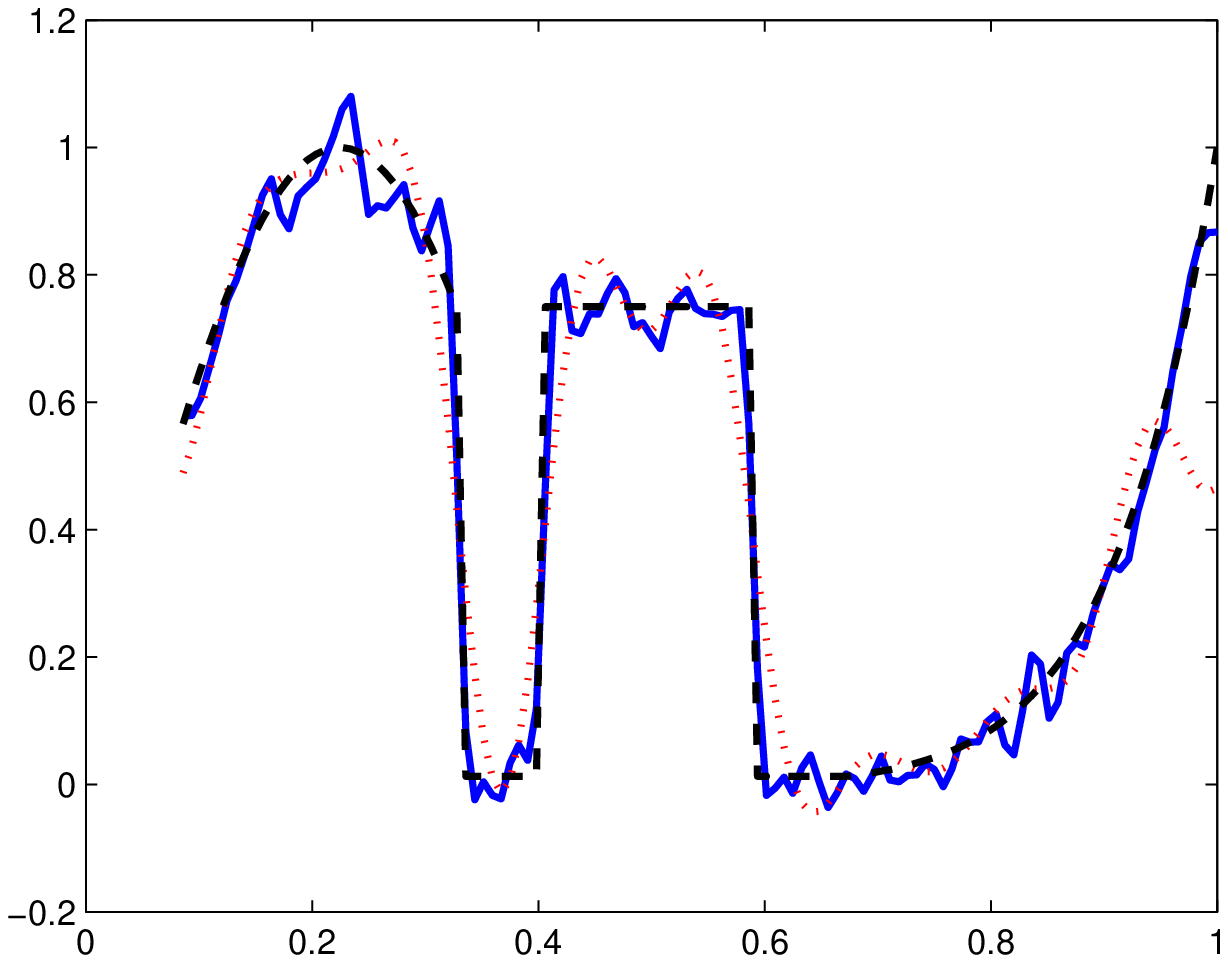}
& \includegraphics[width = 5cm]{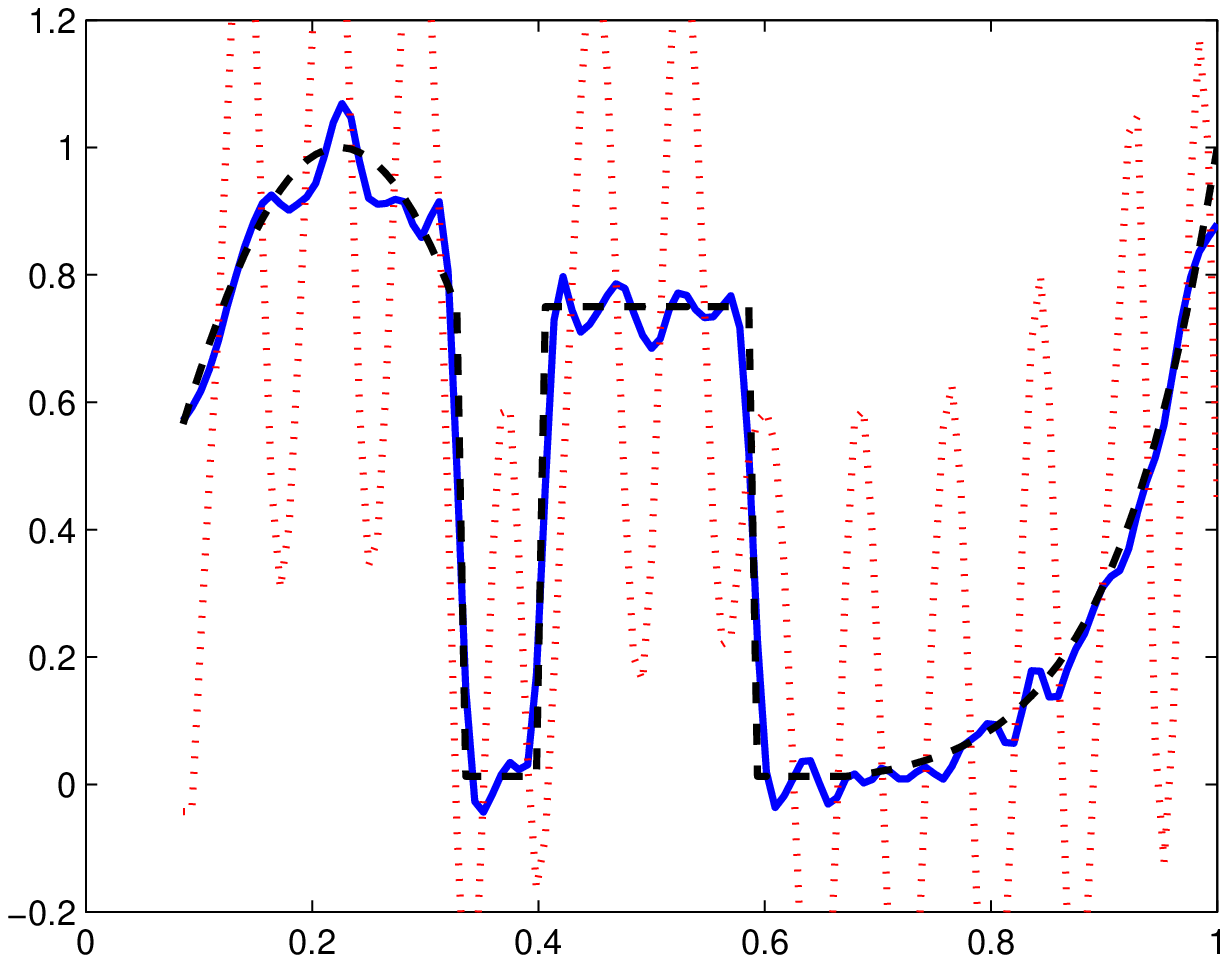}\\
$\mu_{\rm opt}$ & $\mu_{\rm GCV}$
\end{tabular}
\caption{Restored signals:
        --- $A_F$, $\cdot \cdot \cdot$ periodic, - - - true signal.}
\label{fig:rest-1Dnons}
\end{figure}

\section{The multidmensional case}\label{sec:mD}
A standard way for defining the multidimensional transform is by tensor product.
Thus
\[T_X^{(d)}=T_{X,\mi{n}}^{(d)}=T_{X,{n_1}}\otimes \cdots \otimes T_{X,{n_d}}\]
$d$ times, where $\mi{n}=(n_1,\dots,n_d)$ and $T_{X,m}$ is the transform $T_X$
of order $m$.
For a 2D array of size $n\times m$, this is easily implemented
doing $m$ 1D transforms of size $n$ for each column and then
$n$ 1D transforms of size $m$ for each row.

The computation of the eigenvalues is more involved.
The strategy is the same described in \cite{ADS08} for computing the eigenvalues of
antireflective matrices. The algorithm in Section 3.2.1 in \cite{ADS08}
can be applied to our proposal, by replacing the discrete
sine transform by the cosine or the Fourier transform.
More in details, in the 2D case:
\begin{enumerate}
  \item Compute two 1D PSF summing the rows and the columns of the 2D PSF.
  \item Apply two 1D transforms $T_X$ for computing the eigenvalues that correspond
        to the frequencies indexed as the edges of the image (the vertical edges are
        associated to the PSF obtained summing the columns and the horizontal edges to the other PSF).
  \item Apply a 2D cosine or Fourier transform for computing the eigenvalues
        indexed as the inner part of the image.
\end{enumerate}


\subsection{Image deblurring}
We consider the deblurring problem with an out of focus blur and $0.1\%$ of Gaussian noise.
The true and the observed images are shown in Figure \ref{fig:2d}.
The restored images are obtained by using the smoothing operator $L=I$.
\begin{figure}
\centering
\begin{tabular}{c@{\hspace{-1cm}}c}
\includegraphics[width = 5.2cm]{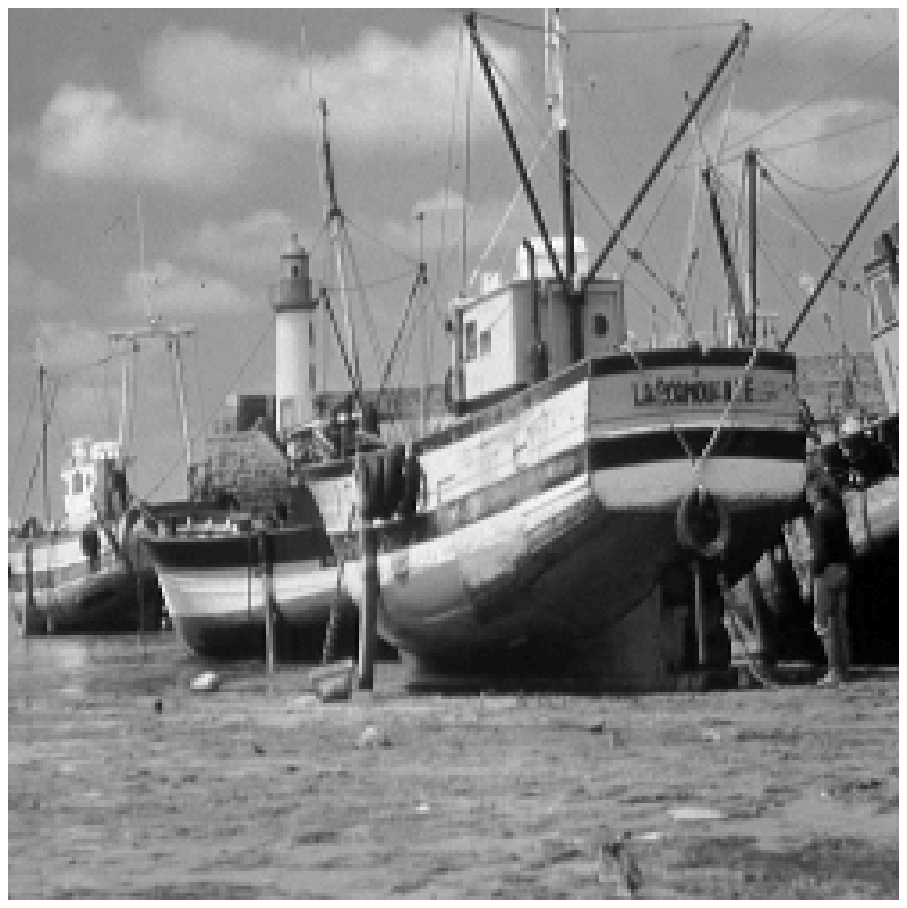}
& \includegraphics[width = 5.2cm]{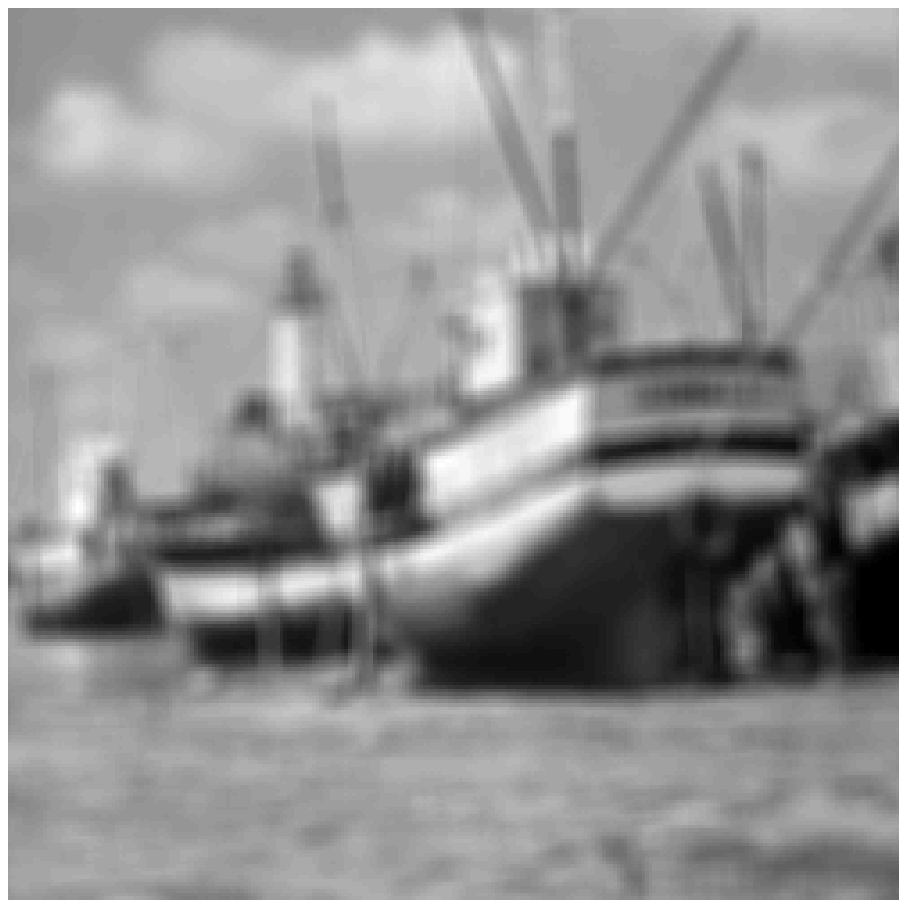}\\
(a) & (b)
\end{tabular}
\caption{(a) True image. (b) Observed image with out of focus blurring and $0.1\%$ of Gaussian noise.}
\label{fig:2d}
\end{figure}

We note that $A_C$ gives a better restoration with respect to antireflective and reflective
boundary conditions (see Figure \ref{fig:rre2d}).
\begin{figure}
\centering
\includegraphics[width = 6cm]{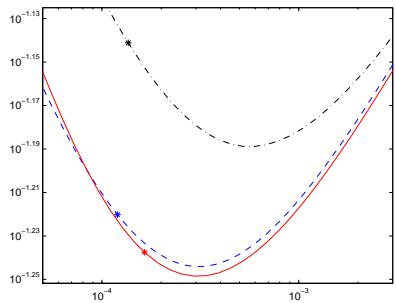}
\caption{RRE: --- $A_C$, - - - antireflective, - $\cdot$ - reflective
    ($*$ denotes values corresponding to $\mu_{\rm GCV}$).}
\label{fig:rre2d}
\end{figure}
In Table \ref{tab:rre-2d} the RRE is shown for $\mu_{\rm opt}$ and $\mu_{\rm GCV}$,
while in Figures \ref{fig:opt-2d} and \ref{fig:gcv-2d} we have the restored images
for the considered boundary conditions and the two choices of $\mu$.

\begin{table}
\caption{RRE for the restoration of the observed image in Figure \ref{fig:2d}.}
\label{tab:rre-2d}
\begin{tabular}{lll}
\hline\noalign{\smallskip}
   & $\mu_{\rm opt}$ & $\mu_{\rm GCV}$ \\
\noalign{\smallskip}\hline\noalign{\smallskip}
reflective & 0.0647 & 0.0723 \\
antireflective & 0.0570 & 0.0602 \\
$A_C$ & 0.0564 & 0.0579 \\
\noalign{\smallskip}\hline
\end{tabular}
\end{table}
\begin{figure}
\centering
\begin{tabular}{@{\hspace{-1cm}}c@{\hspace{-1cm}}c@{\hspace{-1cm}}c}
\includegraphics[width = 5.2cm]{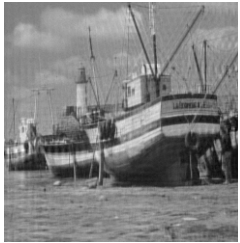}
& \includegraphics[width = 5.2cm]{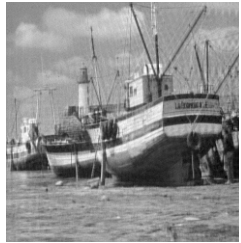}
& \includegraphics[width = 5.2cm]{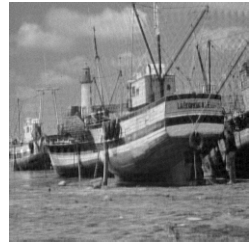}\\
(a) reflective & (b) antireflective & (c) $A_C$
\end{tabular}
\caption{Restored images for $\mu_{\rm opt}$.}
\label{fig:opt-2d}
\end{figure}
\begin{figure}
\centering
\begin{tabular}{@{\hspace{-1cm}}c@{\hspace{-1cm}}c@{\hspace{-1cm}}c}
\includegraphics[width = 5.2cm]{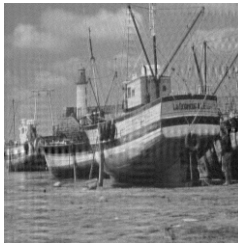}
& \includegraphics[width = 5.2cm]{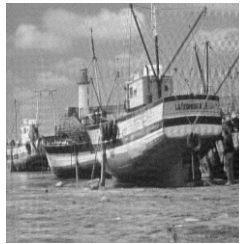}
& \includegraphics[width = 5.2cm]{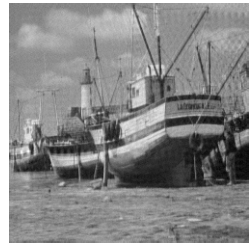}\\
(a) reflective & (b) antireflective & (c) $A_C$
\end{tabular}
\caption{Restored images for $\mu_{\rm GCV}$.}
\label{fig:gcv-2d}
\end{figure}

Even if there is not a large reduction of the RRE, the images restored with
$A_C$ show lesser ringing effects with respect to the antireflective boundary conditions
at least in the south-west corner of the image.

For a general image the use of $A_C$ instead of antireflective boundary
conditions leads to negligible improvement if the image is not smooth enough at
the boundary or if the noise level is so high to dominate the approximation
error in the restoration.

For concluding, we consider a nonsymmetric PSF. The observed image in Figure \ref{fig:nons-2d} (a)
is affected from an out of focus combined with a moving blur. Since the PSF is nonsymmetric,
we compare $A_F$ with periodic boundary conditions like in Section \ref{sec:numexp}.
In Figure \ref{fig:nons-2d} (b) the RRE for $A_F$ is significantly lower than the RRE of
periodic boundary conditions. Indeed, in Figure \ref{fig:rest-2dnons} it is possible to note
that in the case of periodic boundary conditions,
the ringing effects at the edges (in the direction of the motion) damage completely the
restoration also for $\mu_{\rm opt}$.
Moreover, the GCV gives a good estimation of the regularization parameter only
in the case of $A_F$ as it is evident in the plot of Figure \ref{fig:nons-2d} (b).
\begin{figure}
\centering
\begin{tabular}{cc}
\includegraphics[width = 5.2cm]{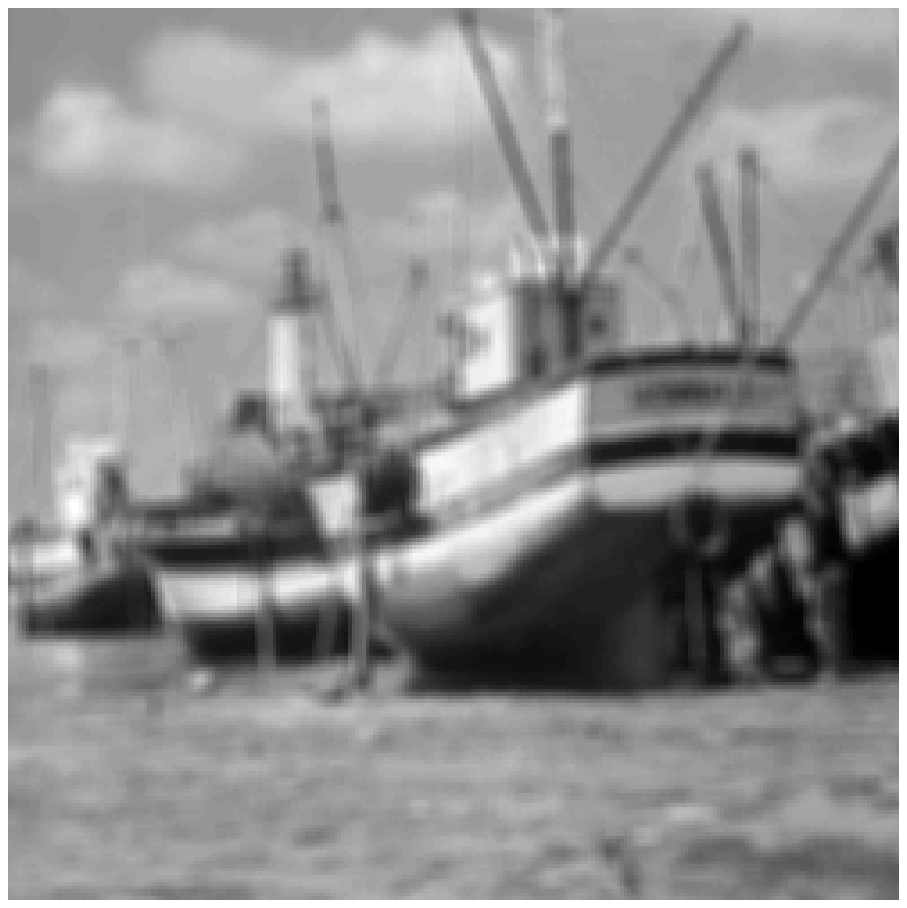}
& \includegraphics[width = 6cm]{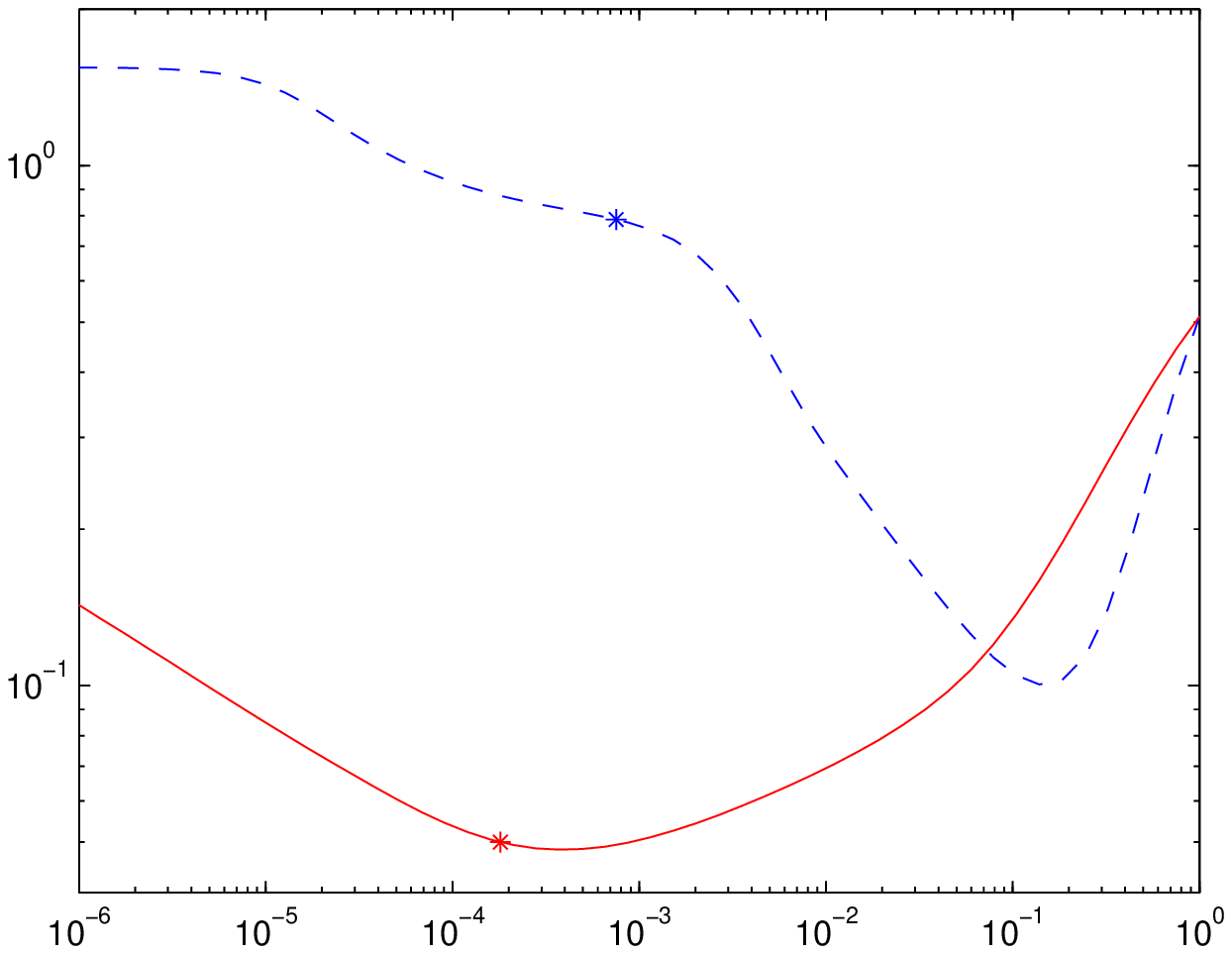}\\
(a)&(b)
\end{tabular}
\caption{(a) Observed image with a nonsymmetric PSF. (b)
RRE: --- $A_F$, - - - periodic boundary conditions
($*$ denotes values corresponding to $\mu_{\rm GCV}$).}
\label{fig:nons-2d}
\end{figure}
\begin{figure}
\centering
\begin{tabular}{@{\hspace{-1cm}}c@{\hspace{-1cm}}c@{\hspace{-1cm}}c}
\includegraphics[width = 5.2cm]{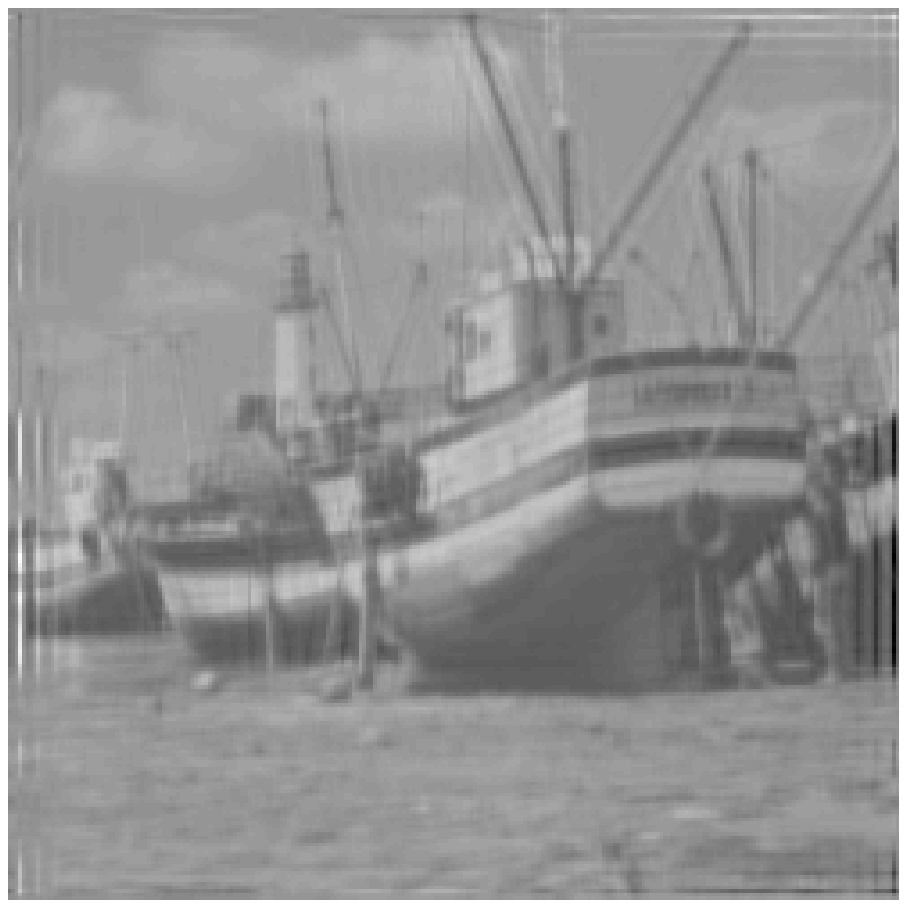}
& \includegraphics[width = 5.2cm]{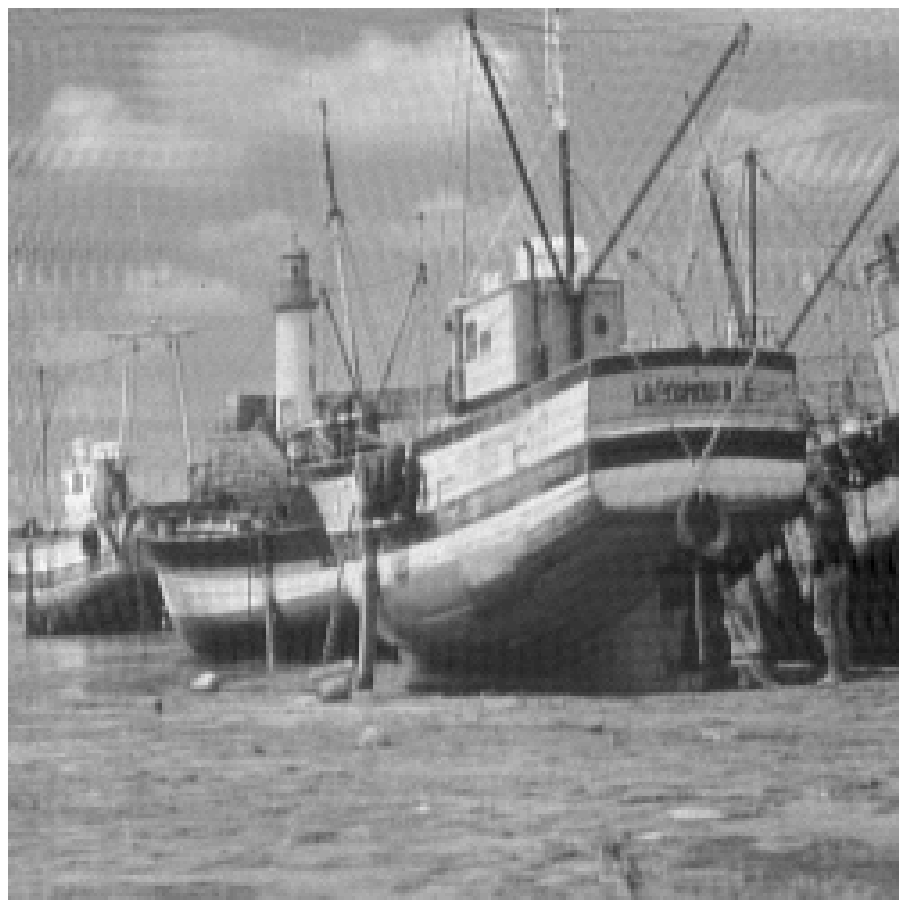}
& \includegraphics[width = 5.2cm]{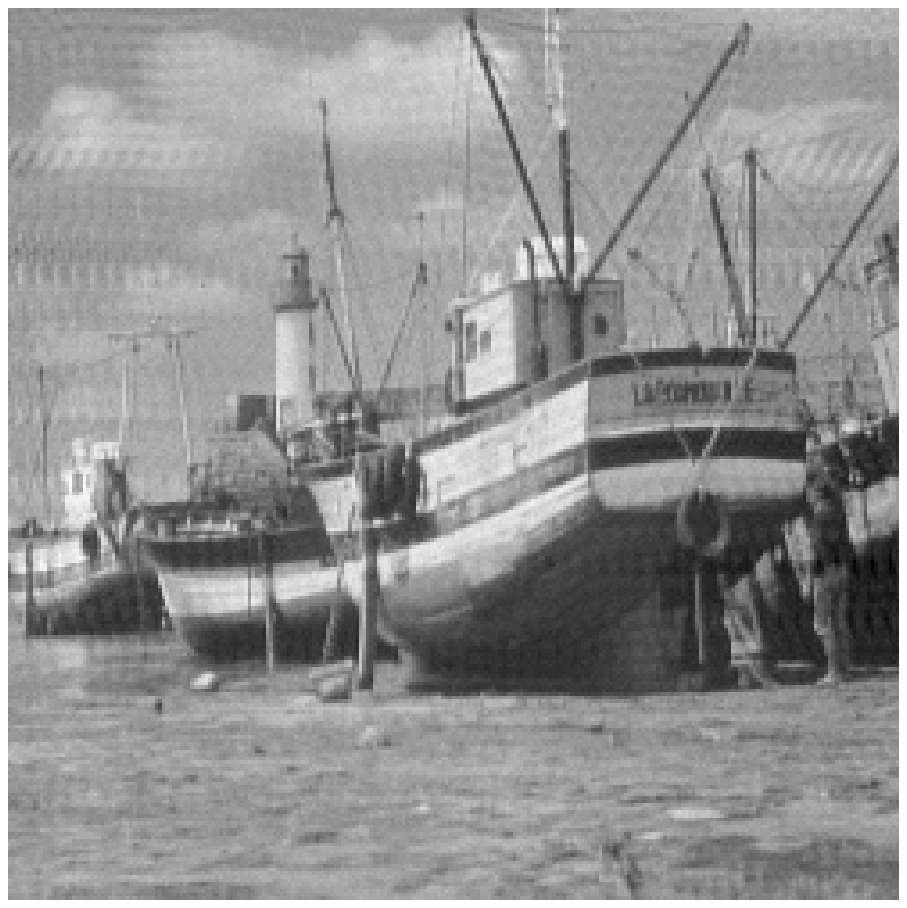}\\
(a) periodic with $\mu_{\rm opt}$ & (b) $A_F$ with $\mu_{\rm opt}$ & (c) $A_F$ with $\mu_{\rm GCV}$
\end{tabular}
\caption{Restored images for the observed image in Figure \ref{fig:nons-2d} (a).}
\label{fig:rest-2dnons}
\end{figure}

\section{Conclusions}\label{sec:concl}
In Section \ref{sec:hord} we have given a framework to construct precise models for
deconvolution problems using fast trigonometric transforms.
The same idea could be applied to different problems having a shift invariant kernel.
Indeed, if we have information on the signal to restore, the set $\S_l$ can be replaced
by other functional spaces that we want to preserve. Moreover, higher order boundary
conditions can be constructed, even if the numerical results show that for
image deblurring problems this approach does not give substantial improvements.

The introduced fast transforms was applied in connection with Tikhonov regularization
and the reblurring approach. However, they could be useful also for more sophisticated
regularization methods like Total Variation for instance.

The analysis of the Tikhonov regularization in Section \ref{sec:tik}
is useful also for the antireflective boundary conditions.
Indeed, it was not previously considered in the literature the case
of $L\neq I$ and the choice of the regularization parameter $\mu$ using the GCV.

Since the proposed transforms are not orthogonal, they were applied in connection with
the reblurring approach, but the theoretical analysis of the regularizing properties of such
approach exists only in the case of antireflective boundary conditions and
symmetric kernel (see \cite{DH08}).
Therefore, a more detailed analysis, especially in the multidimensional case with a nonsymmetric
kernel, should be considered in the future.

\begin{acknowledgements}
I would thank Serra Capizzano for useful discussions.
\end{acknowledgements}



\begin{thebibliography}{biblio}
%
%
\bibitem{ADNS}
    A. Aric\`{o}, M. Donatelli, J. Nagy, and S. Serra Capizzano,
    The Anti-Reflective Transform and Regularization by Filtering,
    Numerical Linear Algebra in Signals, Systems, and Control.,
    in Lecture Notes in Electrical Engineering, edited by S. Bhattacharyya,
    R. Chan, V. Olshevsky, A. Routray, and P. Van Dooren, Springer Verlag, in press.
%
\bibitem{ADS08}
    A. Aric\`o, M. Donatelli, and S. Serra-Capizzano,
    Spectral analysis of the anti-reflective algebra,
    Linear Algebra Appl., 428, 657--675 (2008).
%
%
\bibitem{ChHa08}
    M. Christiansen and M. Hanke,
    Deblurring methods using antireflective boundary conditions,
    SIAM J. Sci. Comput., 30, 855--872 (2008).
%
\bibitem{DES}
    M. Donatelli, C. Estatico, A. Martinelli, and S. Serra Capizzano,
    Improved image deblurring with anti-reflective boundary conditions and re-blurring,
    Inverse Problems, 22, 2035--2053 (2006).
%
\bibitem{DH08}
    M. Donatelli and  M. Hanke,
    On the condition number of the antireflective transform,
    manuscript (2008).
%
\bibitem{AR-reblur}
    M. Donatelli and S. Serra Capizzano,
    Anti-reflective boundary conditions and re-blurring,
    Inverse Problems, 21, 169--182 (2005).
%
%
\bibitem{EHN96}
    H.~W.~Engl, M.~Hanke, and A.~Neubauer,
    Regularization of Inverse Problems,
    Kluwer, Dordrecht (1996).
%
\bibitem{GHW79}
    G. Golub, M. Health, and G. Wahba,
    Generalized cross-validation as a method for choosing good ridge parameter,
    Technometrics, 21, 215--223 (1979).
%
\bibitem{Gene}
    G.~H. Golub and C.~F. Van Loan,
    Matrix Computations, third edition,
    The Johns Hopkins University Press, Baltimore (1996).
%
\bibitem{HaNaOL06}
  P.~C.~Hansen, J.~G.~Nagy, and D.~P.~O'Leary,
  Deblurring Images: Matrices, Spectra, and Filtering,
  SIAM, Philadelphia, PA (2006).
%
\bibitem{NCT}
    M. Ng, R.~H. Chan, and W.~C. Tang,
    A fast algorithm for deblurring models with Neumann boundary conditions,
    SIAM J. Sci. Comput., 21, 851--866 (1999).
%
\bibitem{Perr}
    L. Perrone,
    Kronecker Product Approximations for Image Restoration with Anti-ReflectiveBoundary Conditions,
    Numer. Linear Algebra Appl., 13(1),1--22 (2006).
%
\bibitem{model-tau}
    S. Serra Capizzano,
    A note on anti-reflective boundary conditions and fast deblurring models,
    SIAM J. Sci. Comput. 25(3), 1307--1325 (2003).
\end{thebibliography}
\end{document}